\documentclass{article}
\usepackage{amsmath, amssymb, amsthm, setspace, enumerate}

\newcommand{\perm}{M^\lambda}
\newcommand{\Yng}{Y^\mu}

\newcommand{\sym}{\Sigma_{r}}
\newcommand{\grpalg}{K\left[\sym\right]}
\newcommand{\ynggrp}{\Sigma_\lambda}

\newcommand{\pro}{e_{m,g}}
\newcommand{\End}{{\rm End}_{\grpalg}}
\newtheorem{definition}{Definition}[section]
\newtheorem{lemma}[definition]{Lemma}
\newtheorem{example}[definition]{Example}
\newtheorem{theorem}[definition]{Theorem}

\newtheorem{corollary}[definition]{Corollary}

\begin{document}
\title{Endomorphism Rings of Some Young Modules}
\author{Jasdeep Singh Kochhar}
\date{}
\maketitle
\begin{abstract}
Let $\sym$ be the symmetric group acting on $r$ letters, $K$ be a field of characteristic 2, and $\lambda$ and $\mu$ be partitions of $r$ in at most two parts. Denote the permutation module corresponding to the Young subgroup $\ynggrp$, in $\sym$, by $\perm$, and the indecomposable Young module by $\Yng$. We give an explicit presentation of the endomorphism algebra $\End(\Yng)$, using the idempotents found by Doty, Erdmann and Henke in \cite{DEH}. 
\end{abstract}
\section{Introduction}
Permutation modules of symmetric groups, arising from actions on set partitions, are of central interest in the representation theory of symmetric groups. They also provide a link with the representation theory of general linear groups, via Schur algebras.

Let $K$ be a field of prime characteristic $p$, and let $n$ and $r$ be positive integers. For each partition $\lambda$  of $r$ with at most $n$ parts, let $\perm$ be the permutation module of the symmetric group $\sym$ of degree $r$, corresponding to the set partition $\lambda$. The indecomposable summands of the modules $\perm$ are known as Young modules, where $\Yng$ is the unique summand of $\perm$ that contains $S^\mu$. The module $\perm$ is in general a direct sum of Young modules $\Yng$, and if $\Yng$ occurs as a direct summand of $\perm$, then $\mu \ge \lambda$ in the dominance order of partitions. The $p$-Kostka number, $[\perm: \Yng]$, is the number of indecomposable summands of $\perm$ isomorphic to $\Yng$, and therefore $$\perm \cong  \bigoplus_{\mu \geq \lambda} [M^{\lambda}:Y^{\mu}]Y^{\mu}.$$

The paper \cite{DEH} studies the endomorphism algebra of $\perm$, denoted $S_K(\lambda)$, when $K$ has characteristic 2 and $\lambda$ has at most two parts. In this case, $S_K(\lambda)$ is commutative, and its primitive idempotents are unique. The main result of \cite{DEH} is the explicit construction of all primitive idempotents of $S_K(\lambda)$, establishing a one-to-one correspondence with the 2-Kostka numbers, explicitly: the idempotent corresponding to $[\perm:\Yng]$ generates the endomorphism algebra of the Young module, $\Yng$. 

In this paper, we study the endomorphism algebra of $\Yng$ as the subalgebra of $S_K(\lambda)$ that is generated by the primitive idempotent constructed in \cite{DEH}. We show that the algebra structure of $\End(\Yng)$ depends only on its dimension as a $K$-vector space, but not on the partition $\mu$. The dimension of the endomorphism algebra of $\Yng$, where $\mu$ is a partition of $r$ in at most two parts, is known (see \cite{H}). 

We will give an explicit presentation of $\End(\Yng)$ by giving generators and relations for this algebra. For a precise description, see Theorem \ref{dim thm}. The result of Theorem \ref{dim thm} may be surprising since the submodule structure of these Young modules can get more and more complicated for large $r$, as can be seen for example in \cite{EH}. As a representative for a $t$-dimensional endomorphism algebra, one can take the endomorphism algebra of $Y^{(t-1,t-1)}$; we see in Example \ref{indecomp} that this module is isomorphic to $M^{(t-1,t-1)}$.
\section{Background}
Doty and Giaquinto found presentations of the Schur algebras $S_K(n,r)$ in terms of the universal enveloping algebras of the Lie algebras {\sf gl}$_n$. We assume that $n=2$, which is the case when $\lambda$ is a partition of $r$ in at most two parts \textit{i.e.} $\lambda = (\lambda_1,\lambda_2)$. Based on the results in \cite{DG}, the paper \cite{DEH} determines a basis and a multiplication formula for the endomorphism algebra of $\perm$. We will summarise what we need, for details we refer to \cite{DEH}.
\subsection{The canonical basis for $S_K(\lambda)$}\label{can basis}
\begin{definition} \normalfont Let $K$ be a field of arbitrary characteristic $p$ \textit{i.e.} $p \ge 0$. The algebra $S_K(\lambda):= \End(\perm)$ has basis 
$$\lbrace b(i) : i \in \mathbb{Z} \text{ and } 0 \le i \le \lambda_2\rbrace.$$ We will refer to this basis as the canonical basis of $S_K(\lambda)$ and the multiplication of these basis elements is given by:
$$b(i) \cdot b(j) = \displaystyle\sum_{k=0}^{i} \binom{j+k} {i}\binom{j+k} {k}\binom{m+j+i}{i-k}b(j+k),$$
where $m:= \lambda_1 - \lambda_2$, and we set $b(a)=0$ for $a>\lambda_2.$ Here the coefficients are taken modulo $p$ when the field $K$ has characteristic $p>0$.
\end{definition}
\subsection{Notation}
For an integer $a$ with $p$-adic expansion $a = \sum_{i=0}^{s} a_ip^i$, where $0 \le a_i \le p-1$ for all $i$, we write $a= [a_0,a_1,\ldots,a_s]$. We also have for non-negative integers $m$ and $n$, with respective $p$-adic expansions $m=[m_0,m_1,\ldots,m_s]$ and $n = [n_0,n_1,\ldots,n_t]$, where $s, t \ge 0$, that: $$\binom{m} {n} \equiv_p \prod_{i=0}^{\max\lbrace s,t\rbrace} \binom{m_i}{n_i}.$$ We refer to the right hand side of the above as the $p$-adic expansion of the binomial coefficient.

In a field $K$ of positive characteristic $p$, the following holds:
\begin{lemma}\cite[Lemma 3.7]{DEH}
Let $i = [i_0,i_1,\ldots,i_s]$. Then $b(i) = \prod_{t=0}^s b(i_t\cdot p^t)$ in $S_K(\lambda)$.
\end{lemma}
It can then be proved that the algebra $S_K(\lambda)$ can be generated by the elements $b(p^0), b(p^1),\ldots,b(p^t)$, where $t$ is the unique natural number such that $p^{t}\le \lambda_2 < p^{t+1}$. For the case when $p=2$, the result is immediate; for $i$ with binary expansion $[i_0,i_1,\ldots,i_s]$ the coefficients $i_t$, where $0 \le t \le s$, are 0 or 1.
\subsection{The idempotents $\pro$}
From now on, we assume $K$ is a field of characteristic 2. Let $\lambda$ and $\mu$ be partitions of $r$ in at most two parts \textit{i.e.} $\lambda = (\lambda_1,\lambda_2)$ and $\mu=(\mu_1,\mu_2)$, where $\mu \ge \lambda$ in the dominance order of partitions. Define $m:= \lambda_1-\lambda_2$ and $g:=\lambda_2-\mu_2$, and so given $r$, from $m$ and $g$, we can completely determine $\lambda$ and $\mu$. It is known (see \cite{H}) that $\Yng$ is a direct summand of $\perm$ if and only if $$B(m,g) := \binom{m+ 2g} {g} \not\equiv_2 0.$$In \cite{DEH}, the binary expansion of $B(m,g)$ is used to construct an element of $S_K(\lambda)$, denoted $\pro$. We begin by defining the index sets $I_{m,g}$ and $J_{m,g}$ as follows: $$I_{m,g}:=\lbrace u: g_u = 0 \text{ and } (m+2g)_u=1 \rbrace,$$ $$J_{m,g}:= \lbrace u: g_u=1 \text{ and } (m+2g)_u = 1\rbrace.$$ For a natural number $t$, define elements of $S_K(\lambda)$ by: $$\pro := \displaystyle\prod_{u \in J_{m,g}} b(2^u) \prod_{u\in I_{m,g}}(1-b(2^u)),$$ $$e_{{m,g}_{\le t}} := \prod_{u \in J_{m,g}, u \le t}b(2^u) \prod_{u\in I_{m,g}, u \le t}(1-b(2^u)).$$ If $u$ is contained in $I_{m,g} \cup J_{m,g}$, we say that $b(2^u)$ is \textit{involved} in $\pro$. We can form a correspondence between the factors of $\pro$ and the binomial coefficients that are factors in the binary expansion of $B(m,g)$, as follows: 
$$\begin{array}{c|c|c|c|c}
 \binom{(m+2g)_u}{g_u} & \binom{1}{1} & \binom{1}{0} & \binom{0}{0} & \binom{0}{1}\\
\hline
\text{Factor of } e_{m,g} & b(2^u) & 1-b(2^u) & 1 & 0 \\
\end{array}.$$ Therefore $\pro$ is equal to 0 if and only if $\binom{0}{1}$ is a factor in the binary expansion of $B(m,g)$. This happens if and only if $B(m,g)$ equals 0 modulo 2. By \cite{H}, this is precisely the case when $\Yng$ is not a summand of $\perm$. In \cite{DEH}, it is proved that the  $\pro$ are the primitive orthogonal idempotents in $S_K(\lambda)$, $i.e.$ the following holds:
\begin{theorem}[Idempotent Theorem]\cite{DEH}\label{idempotent theorem} Fix $m\ge 0$. The set of elements $\pro$, with $B(m,g) \neq 0$ modulo 2, and $m+2g \le r$ give a complete set of primitive orthogonal idempotents for $S_K(\lambda)$. \end{theorem} Theorem \ref{idempotent theorem} then implies: \begin{theorem}\cite[Theorem 7.1]{DEH} Let $\lambda = (\lambda_1, \lambda_2)$ and $\mu = (\mu_1, \mu_2)$ be partitions of $r$, such that $\Yng$ is a direct summand of $\perm$. Define $$m:= \lambda_1-\lambda_2 \textit{ and } g:=\lambda_2 - \mu_2.$$ Then $\pro$ is the idempotent element of $S_K(\lambda)$, such that $\pro\perm = \Yng$.\end{theorem} We also recall the following lemma from \cite{DEH}, which is used when finding a minimal set of generators of $\End(\Yng)$. \begin{theorem}[Orthogonality Lemma]\cite{DEH} Suppose that $\binom{(m+2g)_s}{g_s} = \binom{0}{0}$, then $e^2_{m,g}\cdot b(2^s)^2 = 0.$\end{theorem}
\section{The algebra $\End(\Yng)$}
In this section, we will see that the generators of the endomorphism algebras of $\End(\Yng)$ have a notion of \textit{size}. This will determine the elements of the algebra that are zero. Letting \textit{\underline{k}} $= \lbrace 1, 2, \ldots, k \rbrace$, we introduce the following definitions:
\begin{definition}\normalfont Let $A$ be a commutative algebra with fixed generators $$\lbrace x_1, \ldots, x_k\rbrace,$$ such that these generators have square zero. Let  $\emptyset \neq I \subset \textit{\underline{k}}$ and $x := \prod_{i \in I} x_i$, so that $x$ is a monomial in the generators. We also require that $x$ has no repeated factors. Define the function $\phi$ as follows: $$\phi(x) = \phi\left(\displaystyle\prod_{i \in I} x_i\right) := \displaystyle\sum_{i\in I} 2^{i}.$$ \end{definition}

\begin{definition}\label{order}\normalfont Let $A$ be a commutative algebra with fixed generators $$\lbrace x_1, \ldots, x_k\rbrace,$$ such that these generators have square zero. Let  $I$ and $J$ be non-empty subsets of \textit{\underline{k}}, and define $x := \prod_{i \in I} x_i$ and $y := \prod_{j \in J} x_j$, so that $x$ and $y$ are monomials in the generators. We again require that $x$ has no repeated factors and $y$ has no repeated factors. Define the ordering $\preceq$ on such elements $x$ and $y$ as follows: $$x \preceq y \mbox{ if and only if } \phi(x) \le \phi(y),$$ with $\phi$ as in Definition 3.1. One can see that for a fixed $k$, this is a total order, and we define $\vert x \vert$ as the position of $x$ in the ascending chain in this total order. We give an example of this below:\end{definition} 

\begin{example}\normalfont Let $A$ be a commutative algebra with fixed generators $$\lbrace x_1, x_2, x_3 \rbrace,$$ such that these generators have square zero. Then the distinct products in these three generators with no repeated factors are given by the set $$\lbrace x_1,x_2,x_3,x_1x_2,x_1x_3,x_2x_3,x_1x_2x_3\rbrace.$$ From Definition 3.1, we obtain $$\phi (x_1) = 2, \phi(x_2) = 4, \phi(x_3) = 8,$$ $$\phi(x_1x_2) = 6, \phi(x_1x_3) = 10, \phi(x_2x_3) = 12 \text{, and } \phi(x_1x_2x_3) = 14.$$ Then from Definition \ref{order}, we have: $$x_1 \preceq x_2 \preceq x_1x_2 \preceq x_3 \preceq x_1x_3 \preceq x_2x_3 \preceq x_1x_2x_3,$$ and so for example we write $\vert x_2x_3 \vert = 6.$
\end{example}

Let $\lambda = (\lambda_1, \lambda_2)$ and $\mu = (\mu_1,\mu_2)$ be partitions of $r$, and $K$ be a field of characteristic 2. We have that the algebra $\End(\Yng)$ is generated by the non-zero elements of the set $$\lbrace \pro b(2^k) : 1 \le 2^k \le \lambda_2 \rbrace,$$ where $\pro$ is the idempotent of $S_K(\lambda)$ such that $\pro\perm = \Yng$. To find a minimal set of generators, we prove the following lemma:
\begin{lemma}[Involvement Lemma]
If $b(2^i)$ is involved in $\pro$, then either $\pro b(2^i) = \pro$ or $\pro b(2^i) = 0$. 
\end{lemma}
\begin{proof}
We prove this for all $b(2^i)$ by induction on $i$. 

Assume first that $i=0,$ and let $m=[m_0,\ldots,m_s]$ be the binary expansion of $m$. Suppose that $b(2^0) = b(1)$ is involved in $\pro$, then $(m+2g)_0 = 1$. In binary for all $g \ge 0$, we have $(2g)_0 = 0$ and so $m_0 = 1$. By \cite[Example 4.1]{DEH}, we have $b(1)^2 = m_0b(1) = b(1)$. We distinguish two cases: \begin{itemize}
\item If the factor corresponding to $u = 0$ in the definition of $\pro$ is $b(1)$, then for $\pro = x\cdot b(1)$: $$\begin{array}{l l}\pro\cdot b(1) &= x \cdot b(1)^2\\ &= x \cdot b(1)\\ &= \pro. \end{array}$$
\item If the factor corresponding to $u = 0$ in the definition of $\pro$ is $1-b(1)$, then for $\pro = x\cdot (1-b(1))$: $$\begin{array}{l l} \pro\cdot b(1) &= x \cdot (1-b(1)) \cdot b(1)\\ &= x\cdot (b(1)-b(1)) = 0.\end{array}$$
\end{itemize}Therefore the result holds for $i=0$. 

Now let $i > 0,$ and let the result hold for all $b(2^k)$, such that $k<i$. Consider $\pro\cdot b(2^i)$, where $b(2^i)$ is involved in $\pro$. By construction of $\pro$, either $b(2^i)$ is a factor of $\pro$ or $1-b(2^i) \equiv_2 1+b(2^i)$ is a factor of $\pro$. Therefore if, for some $x \in S_K(\lambda)$, we have: $$\pro = x \cdot b(2^i) \mbox{ or } \pro = x \cdot (1+b(2^i)) ,$$ then $$\pro b(2^i) = x \cdot b(2^i)^2 \mbox{ or } x \cdot (b(2^i)+b(2^i)^2),$$respectively. By \cite[Lemma 4.2]{DEH}, for the $v$ such that $0 \le v \le i$ and $v$ is maximal with respect to $m_{v-1} = 0$ in the binary expansion of $m$, one of the following holds:
\begin{itemize}
\item The factor corresponding to $u=i$ in the definition of $\pro$ is $b(2^i)$: As $$b(2^i)^2 = b(2^i)[m_i\cdot 1 + \displaystyle\sum_{k=v-1}^{i-1}b(2^k)^2],$$ we have: $$\begin{array}{l l}\pro b(2^i) &= x \cdot b(2^i)^2\\ &= x[b(2^i)[m_i\cdot 1 + \sum_{k=v-1}^{i-1} b(2^k)^2]]\\ &= \pro [m_i\cdot 1 + \sum_{k=v-1}^{i-1} b(2^k)^2]. \end{array}$$
\item  The factor corresponding to $u=i$ in the definition of $\pro$ is $1+b(2^i)$: As $$b(2^i)+b(2^i)^2 = b(2^i)[1+m_i\cdot 1 + \displaystyle\sum_{k=v-1}^{i-1}b(2^k)^2],$$ we have: $$\begin{array}{l l}\pro b(2^i) &= x \cdot (b(2^i)+b(2^i)^2)\\ &= x[b(2^i)[1+m_i\cdot 1 + \sum_{k=v-1}^{i-1} b(2^k)^2]]\\ &= \pro [1+m_i\cdot 1 + \sum_{k=v-1}^{i-1} b(2^k)^2].\end{array}$$
\end{itemize} Note that if $\pro b(2^i) =\pro$, then $\pro b(2^i)^2=\pro$, and if $\pro b(2^i) = 0$, then $\pro b(2^i)^2=0$. 

By the induction hypothesis, for $v-1 \le k \le i-1$ such that $b(2^k)$ is involved in $\pro$, either $\pro b(2^k)^2 = \pro$ or $\pro b(2^k)^2 = 0$. For $b(2^k)$ not involved in $\pro$, the factor $\binom{(m+2g)_k}{g_k}$ of the binary expansion of $\binom{m+2g}{g}$ satisfies the condition of the Orthogonality Lemma. Therefore in all cases: $$\pro b(2^i) = \pro \cdot (\text{sum of 1's and 0's}).$$ This is equal to $\pro$ or 0 as we are in a field of characteristic 2, and so the result holds by induction.
\end{proof} It follows from the Involvement Lemma that the generators of the algebra $\End(\Yng)$ are the $\pro b(2^i)$ such that $b(2^i)$ is not involved in $\pro$. These are precisely the non-zero elements of the set \begin{equation}\label{eq: gens}T:= \lbrace\pro b(2^s) : (m+2g)_s = 0 \rbrace.\end{equation} As $\binom{m+2g}{g}$ is non-zero, the generators of the algebra $\End(\Yng)$ satisfy the conditions of the Orthogonality Lemma. It follows that the generators of $\End(\Yng)$ all have square zero, and so for $k$ equal to the cardinality of the set $T$, the algebra $\End(\Yng)$ is isomorphic to a quotient of \begin{equation}\label{eq: orth} K[x_1,\ldots,x_k]/(x_i^2: i = 1, \ldots, k). \end{equation}
\begin{definition}\normalfont Let $i = [i_0, i_1, \ldots, i_s]$. We call the set $$S_i:=\{ u: i_u \neq 0\}$$ the support of the basis element $b(i)$ of $S_K(\lambda)$.
\end{definition} We note that defining the support of $b(i)$ in this way gives a bijection between the canonical basis elements of $S_K(\lambda)$ and the subsets of $\{1, 2, \ldots, \lambda_2\}$, via the map $b(i) \mapsto S_i$. 

We now fix $e_{m,g}$.  Consider $\pro$ as a linear combination of the elements in the canonical basis of $S_K(\lambda)$. By the construction of $\pro$, a basis element $b(i)$ occurs in this linear combination only if it is a product of elements $b(2^u)$ such that $(m+2g)_u=1$, {\it i.e.} if $u$ is an element of $S_i$, then $(m+2g)_u=1$. Therefore the support of $b(i)$ is contained in $I_{m,g}\cup J_{m,g}$. 

\begin{theorem}[Basis Theorem]\label{basis} Let $K$ be a field of characteristic 2. Suppose that $r$ is a postive integer, and let $\mu = (\mu_1, \mu_2)$ be a partition of $r$. The algebra $\End(\Yng)$ has basis given by the non-zero elements in the set
$$S := \{ e_{m,g}b(j):  \mbox{the support of $b(j)$ is disjoint from $I_{m,g}\cup J_{m,g}$}\},$$
where $m$ and $g$ are such that $\pro \perm = \Yng$.
\end{theorem}
\begin{proof}
We first show that the set $S$ spans the algebra $\End(\Yng)$. Suppose that $j$ is contained in $I_{m,g} \cup J_{m,g}$. Then by the Involvement Lemma, we have that either 
\begin{center} $\pro b(2^j)= \pro$ or $\pro b(2^j) = 0$. \end{center} As $b(i)b(0) = b(i)$ and hence $\pro = \pro b(0)$, in both of the above cases $\pro b(2^j)$ can be expressed as a linear combination of the elements in the set $S$.

Consider a non-zero element $e_{m,g}b(i)$ of $\End(\Yng)$. If we write \\$\pro b(i)$ as a linear combination of the elements in the canonical basis of $S_K(\lambda)$, then all $b(l)$ that occur in this linear combination can be written as:
\begin{equation}\label{eq: decomp} b(l) = b(k)b(j).\end{equation} Using the binary expansion of $l$, we construct this factorisation such that the basis element $b(k)$ has support contained in $I_{m,g}\cup J_{m,g}$ and disjoint from the support of $b(j)$. By construction, the support of $b(l)$ is then the disjoint union of the support of $b(k)$ and the support of $b(j)$. For the $b(l) = b(k)b(j)$ in the linear combination of $\pro b(i)$, as $\pro$ is an idempotent element of $S_K(\lambda)$, we have that $\pro b(i)$ equals the sum of the elements $$\pro b(k)b(j).$$ Repeatedly using the Involvement Lemma and that $\pro$ is an idempotent, we have $\pro b(k) = \pro$ or $\pro b(k) = 0$. Therefore $\pro b(i)$ can be written as a linear combination of the set $S$, and so the elements of $S$ span the algebra $\End(\Yng)$. 

It therefore remains to show that the elements of the set $S$ are linearly independent. From the decomposition of the basis elements described in (\ref{eq: decomp}), it follows that if $i \neq j$ and both $\pro b(i)$ and $\pro b(j)$ belong to the set $S$, then the linear combinations of these two elements in terms of the canonical basis of $S_K(\lambda)$ have no basis elements in common. Hence $S$ is linearly independent.
\end{proof} We now prove the following result:
\begin{theorem}\label{dim thm}
Let $K$ be a field of characteristic 2. Suppose that $r$ is a positive integer and $A= \End(\Yng)$, where $\mu = (\mu_1,\mu_2)$ is a partition of $r$. If $A$ has dimension $n$ and $k$ is the unique non-negative integer such that $2^{k-1} < n \le 2^k$, then $A$ is isomorphic as a $K$-algebra to $$B:=K[x_1,\ldots,x_k]/(\lbrace x_i^2: i = 1,\ldots,k\rbrace \cup R),$$ where $$R:= \lbrace x=x_{r_1}x_{r_2}\ldots x_{r_l}x_k : \mbox{$r_i\neq r_j$ and }\vert x \vert \ge n\rbrace$$ and $\vert x \vert$ is as in Definition \ref{order}.
\end{theorem}
\begin{proof}
Using the discussion on the algebra $A$ leading to (\ref{eq: orth}), the dimension $n$ of $A$ satisfies $n \le 2^k$, where $k$ is the size of the set $T$ in (\ref{eq: gens}). We label the elements in $T$ as: $$\pro b(2^{a_i}), \ a_1< a_2 < \ldots < a_k.$$

Consider an element $\pro b(i) \neq 0$, where $b(i)$ has $b(2^{a_k})$ as a factor.  Such an element is non-zero if and only if the following two conditions hold: 
\begin{enumerate}[(i)]
\item The support of $b(i)$ is disjoint from $I_{m,g} \cup J_{m,g}$.
\item The term $b(l)$ occurring in the expansion of $\pro b(i)$ with smallest $l$ satisfies $l \le \lambda_2$. 
\end{enumerate} By (\ref{eq: gens}), each generator of $\End(\Yng)$ has disjoint support from $I_{m,g} \cup J_{m,g}$. Therefore a product of these generators with no repeated factors, $\pro b(j)$, also does. If $j < i$, then the smallest $b(l)$ occurring in the expansion of $\pro b(j)$ also satisfies $l \le \lambda_2$. Therefore $\pro b(j)$ is non-zero.

As $\pro b(2^{a_k})$ is non-zero, using the labelling $b(2^{a_i}) \mapsto x_i$, the algebra $A$ has a proper subalgebra isomorphic to $$K\left[x_1, \ldots , x_{k-1}\right]/(x^2_1,\ldots , x^2_{k-1}),$$ of dimension $2^{k-1}$. This subalgebra does not contain $x_k$, and so we must have that $2^{k-1} < n$.

The squares of the generators of $A$ are zero, and the multiplication of the generators of $A$ is commutative. We therefore have a well-defined algebra map: $$\theta: K[x_1, x_2, \ldots, x_k]/(x_i^2) \longrightarrow A,$$ where $\theta(x_i) := e_{m,g}b(2^{a_i})$, and this is surjective.

If $e_{m,g}b(i)$ is in the basis $S$ from Theorem \ref{basis}, then $b(i)$ is uniquely a product of distinct $b(2^{a_j})$. We therefore obtain a natural linear ordering on the basis elements of $A$, using the natural order on the integers that are sums of distinct $2^{a_j}$. We note that the map $\theta$ preserves the linear order on the generators of $K\left[x_1,\ldots,x_k\right]$, as defined in Definition \ref{order}. By counting dimensions and using the basis theorem, it follows that the kernel of $\theta$ is the ideal: $$(\lbrace x_i^2: i = 1,\ldots,k\rbrace \cup R),$$ for $$R:= \lbrace x=x_{r_1}x_{r_2}\ldots x_{r_l}x_k : \mbox{$r_i\neq r_j$ and }\vert x \vert \ge n\rbrace,$$ where $\vert x \vert$ is as in Definition \ref{order}, and so the result follows. 
\end{proof} As an immediate consequence of this theorem, we see that for partitions $\gamma = (\gamma_1,\gamma_2)$ and $\mu = (\mu_1,\mu_2)$, and $K$ a field of characteristic 2, the algebras $\End(Y^\gamma)$ and $\End(\Yng)$ are isomorphic if and only if they have the same dimension.
\begin{example}\label{indecomp}
\normalfont Let $\lambda = (n -1, n -1)$. Let $\mu = (\mu_1,\mu_2)$ be a partition of $2(n - 1)$. Writing $g=\lambda_2-\mu_2$, we obtain:
$$\binom{m+2g}{g} = \binom{2g}{g},$$
as $m= \lambda_1-\lambda_2 = 0$. This binomial coefficient is non-zero modulo $2$ if and only if $g=0$. Therefore the identity element of $S_K(\lambda)$ is a primitive idempotent in $S_K(\lambda)$ and $M^{(n-1,n-1)} \cong Y^{(n-1,n-1)}$ by the definition of the Young modules. Therefore by Definition $\ref{can basis}$, the algebra $\End(Y^{(n-1,n-1)})$ has dimension $n$, with basis given by $$\lbrace 1, b(1), \ldots, b(n-1)\rbrace.$$ By (\ref{eq: gens}), the generators of $\End(Y^{(n-1,n-1)})$ are $$\lbrace b(2^i) : 1 \le 2^i \le n-1 \rbrace.$$ Labelling these generators as \begin{equation}\label{eq: labelling}b(2^i) \mapsto x_{i+1} \text{ for } i=0,\ldots,k-1,\end{equation} we recall from the proof of Theorem \ref{dim thm} that $\End(Y^{(n-1,n-1)})$ contains a proper subalgebra isomorphic to $$K[x_1,\ldots,x_{k-1}]/( x_i^2 : i = 1, \ldots, k-1 ),$$ of dimension $2^{k-1}$. As $b(2^{k-1})$ is also a basis vector of $\End(Y^{(n-1,n-1)})$, there are only $n-(2^{k-1}+1)$ other possible basis vectors for the endomorphism algebra of $Y^{(n-1,n-1)}$. Using the labelling defined in (\ref{eq: labelling}), the remaining $n-(2^{k-1}+1)$ basis vectors will be products of generators with $x_k$ as a factor. By Definition \ref{can basis}, an element $b(i)$ is zero if and only if $i > n-1$. Therefore the images of the monomials that are zero in $\End(\Yng)$, with $\pro b(2^{a_k})$ as a factor, are given by the set:
$$R:=\lbrace x_{r_1}x_{r_2}\ldots x_{r_l}x_k : \mbox{$r_i\neq r_j$ and }\sum_{i=1}^l 2^{r_i - 1} + 2^{k-1} > n-1 \rbrace, $$ and so $\End(Y^{(n-1,n-1)})$ is isomorphic to the algebra:
$$K[x_1,\ldots,x_k]/(\lbrace x_i^2: i =1, \ldots, k \rbrace \cup  R ).$$
\end{example}
\begin{corollary}[Dimension Theorem] Let $\mu = (\mu_1,\mu_2)$ be a partition of $r$. If the algebra $A = \End(\Yng)$ has dimension $n$, then $A$ is isomorphic to the algebra
$$K[x_1,\ldots,x_k]/(\lbrace x_i^2: i = 1, \ldots, k \rbrace \cup  R),$$ where $R$ is such that $$R = \lbrace x_{r_1}x_{r_2}\ldots x_{r_l}x_k : \mbox{$r_i\neq r_j$ and } \sum_{i=1}^l 2^{r_i - 1} + 2^{k-1} > n-1 \rbrace.$$ 
\end{corollary}
\begin{proof}
This follows from Theorem \ref{dim thm} and Example \ref{indecomp}.
\end{proof}
\begin{center}
{\bf Acknowledgements}
\end{center}
The work in this paper was completed while the author was completing his MSc. thesis, under the supervision of Dr. Karin Erdmann. The author most gratefully acknowledges her support. The final publication is available at Springer via http://dx.doi.org/10.1007/s00013-015-0854-2.
\bibliographystyle{unsrt}

\begin{thebibliography}{9}
\bibitem{DEH} S. Doty, K. Erdmann, and A. Henke, {\it Endomorphism rings of permutation modules over maximal Young subgroups}, J. Algebra {\bf 307} (2007), 377--396.
\bibitem{H} A. Henke, {\it On p-Kostka numbers and Young modules}, {European Journal of Combinatorics} {\bf 26} (2005), 923--942.
\bibitem{EH} K. Erdmann and A. Henke {\it On Schur algebras, Ringel duality and symmetric groups}, {Journal of Pure and Applied Algebra} {\bf 169} (2002), 175--199. 
\bibitem{DG} S. Doty and A. Giaquinto, {\it Presenting Schur algebras as quotients of the universal enveloping algebra of $\mathfrak{gl}_2$}, {Algebras and Representation Theory} {\bf 7} (2004), 1--17.
\end{thebibliography}

\end{document}